\documentclass[11pt,reqno]{amsart}

\usepackage[T1]{fontenc}
\usepackage{lmodern}
\usepackage{microtype}
\usepackage{amsmath,amssymb,amsthm,mathtools,mathrsfs}
\usepackage[a4paper,margin=28mm,headheight=14pt]{geometry}
\usepackage{enumitem}
\usepackage[
  colorlinks=true,
  linkcolor=red,
  citecolor=blue,
  urlcolor=blue
]{hyperref}

\allowdisplaybreaks[2]
\numberwithin{equation}{section}
\setlist[enumerate]{label=\textup{(\roman*)},leftmargin=*,itemsep=3pt}

\newtheorem{theorem}{Theorem}[section]
\newtheorem{proposition}[theorem]{Proposition}
\newtheorem{lemma}[theorem]{Lemma}
\newtheorem{corollary}[theorem]{Corollary}

\theoremstyle{definition}

\theoremstyle{remark}
\newtheorem{remark}[theorem]{Remark}

\newcommand{\tr}{\operatorname{tr}}

\newcommand{\op}{\mathrm{op}}
\newcommand{\C}{\mathbb C}

\title{Pointwise energy monotonicity for stable Higgs bundles}
\author[Tianzhi Hu]{Tianzhi Hu}
\address{
School of Mathematics and Statistics, Wuhan University, Wuhan, Hubei 430072, P.R. China
}
\email{hutianzhi@whu.edu.cn}
\date{}

\begin{document}

\begin{abstract}
We study the Dai--Li conjecture concerning the pointwise monotonicity of the energy density along the $\mathbb C^*$-flow of stable $\mathrm{SL}(n,\mathbb C)$ Higgs bundles. We prove that the conjecture holds in rank two: for every stable $\mathrm{SL}(2,\mathbb C)$ Higgs bundle, the energy density is pointwise nondecreasing along the $\mathbb C^*$-orbit. In contrast, we show that this phenomenon is genuinely rank-dependent. For every rank $n\geq 3$, we construct stable $\mathrm{SL}(n,\mathbb C)$ Higgs bundles for which the energy density fails to be monotone along the $\mathbb C^*$-flow.
\end{abstract}

\maketitle

\section{Introduction}\label{sec:introduction}

Let $X$ be a compact Riemann surface of genus $g\geq 2$.  Non-abelian Hodge theory
\cite{UhlenbeckYau1986,Hitchin1987,Donaldson1987,Corlette1988,Simpson1992} relates
representations of $\pi_1(X)$, equivariant harmonic maps, and Higgs bundles on
$X$.  In particular, a stable degree-zero Higgs bundle $(E,\Phi)$ admits a
harmonic metric $h$, unique up to a positive scalar. The harmonic metric induces a flat bundle structure and an equivariant harmonic map
\[
 f:\widetilde X\longrightarrow G/K
\]
to the corresponding symmetric space.

Fix a conformal metric $g_X$ on $X$ and an invariant metric $g_{G/K}$ on the
symmetric space.  We write the \textbf{pointwise energy density} and the \textbf{total energy} as
\begin{equation}\label{eq:intro-energy-definitions}
 e_{(E,\Phi)}(p)
 :=\operatorname{tr}_{g_X}\bigl(f^*g_{G/K}\bigr)(p),
 \qquad
 \mathcal E(E,\Phi)
 :=\int_X e_{(E,\Phi)}\,\mathrm{vol}_{g_X}.
\end{equation}

The moduli space of Higgs bundles carries the natural $\C^*$-action
\[
 (E,\Phi)\longmapsto(E,\lambda\Phi),\qquad \lambda\in\C^*.
\]
Hitchin's Morse-theoretic analysis of this action shows that the total energy
is nondecreasing as $|\lambda|$ increases and is strictly monotone along a nonconstant orbit
\cite{Hitchin1987,Hitchin1992}.  

Motivated by this, Dai and Li proposed the \textbf{pointwise energy monotonicity conjecture}; see \cite{DaiLiCyclic} and \cite{LiHarmonic}: along the $\C^*$-orbit of a (stable) Higgs bundle,
the energy density should increase pointwise with $|\lambda|$.
Dai and Li proved this monotonicity for stable cyclic Higgs bundles
\cite{DaiLiCyclic}.  A related
pointwise domination result for Higgs bundles in the Hitchin section had been
obtained by Li \cite{LiHarmonic}.  Further pointwise domination results within
$n$-Fuchsian Hitchin fibers were established by Dai and Li
\cite{DaiLiDomination}.  Related domination results for surface group
representations were established by Deroin and Tholozan \cite{DeroinTholozan}.
More recently, Sagman--To\v{s}i\'c proved the Dai--Li monotonicity conjecture for
stable simple Coxeter cyclic $G$ Higgs bundles for every simple complex group
$G$ \cite{SagmanTosic}. 

For the monotonicity problem it is enough to consider the positive real part of
the $\C^*$-action.  Indeed, if we write $\lambda=te^{i\theta}$ with $t>0$, the
harmonic metric and the energy density depend only on $t=|\lambda|$.  We shall
therefore write
\[
 (E,\Phi_t):=(E,t\Phi),\qquad t>0,
\]
and denote the corresponding harmonic metric and energy density by $h_t$ and
$e_t$.

Our main result settles the Dai--Li pointwise energy monotonicity conjecture completely in rank two.  

\begin{theorem}\label{thm:mainSL2}
Let $(E,\bar\partial_E,\Phi)$ be a stable degree-zero $\mathrm{SL}(2,\mathbb C)$ Higgs bundle over $X$. For every
$t>0$, let $h_t$ be the harmonic metric of $(E,t\Phi)$ inducing fixed
determinant metrics, and let $e_t$ be the energy density of $h_t$. Then:
\begin{enumerate}
\item[\textup{(i)}] \emph{(Non-strict part.)} For every $x\in X$ and $t>0$, we have $\partial_t e_t(x)\ge0.$
\item[\textup{(ii)}] \emph{(Rigidity of the non-strict case.)} If
$\Phi(x)\ne0$ and $\partial_t e_t(x)=0$ for some $x\in X$ and some $t>0$, then $(E,\Phi)$ is a $\C^*$-fixed Higgs
bundle. 
\end{enumerate}
\end{theorem}

The proof is reduced to a rank-two estimate for the logarithmic variation of
the harmonic metric.  Define the smooth endomorphism
\[
 \Gamma_t:=t\cdot h_t^{-1}\partial_t h_t.
\]
Because the determinant metric is fixed, $\Gamma_t$ is trace-free; it is also
Hermitian self-adjoint with respect to $h_t$.  The pointwise energy variation
formula obtained in Section~\ref{sec:linearization} shows that the non-strict part of
Theorem~\ref{thm:mainSL2} follows once one proves the following bound.  The
equality analysis of the same estimate yields the rigidity statement.

\begin{proposition}[Metric-variation estimate]\label{prop:Gammabound}
For every $t>0$,
\[
 \|\Gamma_t(x)\|_{\op}\le 1
 \qquad\text{for every }x\in X,
\]
where the operator norm is computed with respect to $h_t$.
\end{proposition}

The mechanism behind Proposition~\ref{prop:Gammabound} is specific to rank two.
The two eigenvalues of the trace-free Hermitian endomorphism $\Gamma_t$ are
$\gamma_t$ and $-\gamma_t$.  The linearized Hitchin equation gives a scalar
Bochner inequality for $\gamma_t$ on the region $\{\gamma_t>1\}$, and the
maximum principle forces $\gamma_t\le1$.  Equality at one point propagates and
produces a parallel grading, which is exactly the $\C^*$-fixed Hodge-bundle
case.  

The rank-two result is sharp with respect to the rank.  In every rank $n\ge3$, the Dai--Li pointwise energy monotonicity conjecture fails for some stable, non-$\C^*$-fixed Higgs bundle.

\begin{theorem}\label{thm:highrank-counterexample}
Let $X$ be any compact Riemann surface of genus $g\ge2$, and let
$n\ge3$.  Then there exist a stable $\mathrm{SL}(n,\C)$ Higgs bundle
$(E_n,\Phi_n)$ which is not fixed by the $\C^*$-action and a point $p\in X$
with $\Phi_n(p)\ne0$ such that, for the normalized harmonic metric $h_{n,t}$
of $(E_n,t\Phi_n)$ and the corresponding energy density $e_{n,t}$, we have
\begin{equation}\label{eq:highrank-limits}
 e_{n,t}(p)
 =a_{0,n}t^{-2(n-2)}+O\!\left(t^{-3(n-2)}\right),
\end{equation}
and
\begin{equation}\label{eq:highrank-derivative-limit}
 \partial_t e_{n,t}(p)
 =-2(n-2)a_{0,n}t^{-2(n-2)-1}
 +O\!\left(t^{-3(n-2)-1}\right)
\end{equation}
for some $a_{0,n}>0$.  In particular,
\[
 e_{n,t}(p)\longrightarrow0,
 \qquad
 \partial_t e_{n,t}(p)<0
\]
for all sufficiently large $t$, whereas $e_{n,1}(p)>0$. 
\end{theorem}

The counterexamples lie entirely in the nilpotent cone and hence are outside the generically regular semisimple regime of the usual large-scale asymptotic theory. They are also consistent with the recent general asymptotic picture of Sagman--Smillie \cite{SagmanSmillie}: after a \(t\)-dependent complex gauge transformation, the large-\(t\) family converges to a nilpotent harmonic bundle. In our construction this gauge transformation and the limiting Hodge bundle are explicit.

Moreover, Li \cite[Conjecture~10.3]{LiSurvey} raised the following conjecture: if $(E,\Phi)$ is a stable $\mathrm{SL}(n,\C)$ Higgs bundle and we denote
\[
 [(E_o,\Phi_o)]
 :=\lim_{t\to0} t\cdot[(E,\Phi)]
\]
in the Higgs-bundle moduli space, then
\begin{align}\label{eq:Liconj}
 e_{(E,\Phi)}(p)\ge e_{(E_o,\Phi_o)}(p)
 \qquad\text{for every }p\in X.
\end{align}

This conjecture was known in several special cases. Li proved it for Higgs bundles in the Hitchin section \cite{LiHarmonic}, while Dai and Li established the stronger pointwise monotonicity along the entire $\C^*$-orbit for stable cyclic Higgs bundles \cite{DaiLiCyclic}.

As an immediate consequence of Theorem~\ref{thm:mainSL2}(i), this conjectured inequality \eqref{eq:Liconj} holds for stable
$\mathrm{SL}(2,\mathbb C)$ Higgs bundles. However, we will use $(E_n,\Phi_n)$ in Theorem~\ref{thm:highrank-counterexample} to produce a counterexample to the above conjecture for any $n\geq 3$; see Section~\ref{subsec:Li-conjecture-10-3}.

Sections~\ref{sec:conventions}--\ref{sec:maximum-principle} prove
Theorem~\ref{thm:mainSL2} and Proposition~\ref{prop:Gammabound}; Section~\ref{sec:highrank-counterexample} proves Theorem~\ref{thm:highrank-counterexample}.

\subsection*{Acknowledgements:} The author would like to thank Qiongling Li for her encouragement and valuable comments on this work.
\subsection*{AI Declaration:} ChatGPT was used for calculations, proving intermediate results, and polishing the writing. All
mathematical statements have been independently checked by the author, who takes full
responsibility for the content.

\section{Conventions and the harmonic metric}\label{sec:conventions}

Fix $t>0$ and set $\Phi_t:=t\Phi.$ Because the determinant metric of $h_t$ is fixed and $\tr\Phi=0$, the rank-two Hitchin equation takes the trace-free form
\begin{equation}\label{eq:Hitchin}
F_{h_t}+[\Phi_t,\Phi_t^{\star_{h_t}}]=0.
\end{equation}
The holomorphic structure is fixed as $t$ varies. Stability is unchanged for $t>0$.

\begin{lemma}[Smooth dependence]\label{lem:smoothdependence}
After the determinant metric is fixed, the harmonic metric $h_t$ depends smoothly on $t\in(0,\infty)$. Such a harmonic metric is called \textbf{normalized}.
\end{lemma}

\begin{proof}
This is the standard smooth-dependence result for stable solutions of Hitchin's equations with fixed determinant; see Hitchin \cite{Hitchin1987}.
\end{proof}

We may therefore set $s:=\log t$, $\partial_s:=t\partial_t$ and \[
\Gamma_t:= h_t^{-1}\partial_s h_t.\]
This $\Gamma_t$ is a smooth endomorphism. Since the determinant metric is fixed,
\[
\tr \Gamma_t=\partial_s\log\det h_t=0.
\]
Moreover $\Gamma_t$ is Hermitian self-adjoint with respect to $h_t$. Since $\Gamma_t$ is also trace-free, its eigenvalues are real and can be written as $\gamma_t$ and $-\gamma_t$, where we order them so that $\gamma_t\ge0$. Thus, at each point, there is an $h_t$-unitary basis in which
\begin{equation}\label{eq:Gammadiag}
\Gamma_t=
\begin{pmatrix}
\gamma_t&0\\
0&-\gamma_t
\end{pmatrix},\qquad \gamma_t\ge0,
\end{equation}

Let $z$ be a local holomorphic coordinate and write
\[
\Phi_t=A_t\,dz.
\]
If $D_{h_t}=D_{h_t}'+D_{h_t}''$ is the Chern connection, then in a local trivialization, we have
\[
D_{h_t}'=\partial+h_t^{-1}\partial h_t,
\qquad D_{h_t}''=\bar\partial.
\]
When writing local coefficient identities below, we suppress the factors $dz$ and $d\bar z$; thus, for example,
\[
D_{h_t}'\Gamma_t
=\partial_z\Gamma_t+[h_t^{-1}\partial_z h_t,\Gamma_t],
\qquad
D_{h_t}''\Gamma_t=\partial_{\bar z}\Gamma_t.
\]
We use $A_t^{\star_{h_t}}$ for the $h_t$-adjoint of $A_t$. Let
\[
\varepsilon_t:=\tr\!\left(A_tA_t^{\star_{h_t}}\right).
\]

The harmonic-map energy density is, up to a fixed positive
normalization constant, the pointwise squared norm of the Higgs field;
see, for instance, \cite[\S2.3, Eq.~(6)]{LiHarmonic}.
Thus, if
\[
g_X=\lambda(z)|dz|^2,\qquad \Phi_t=A_t\,dz,
\]
then
\begin{equation}\label{eq:energynormalization}
e_t(z)
=
C\,\lambda(z)^{-1}
\tr\!\left(A_tA_t^{\star_{h_t}}\right)
\end{equation}
for some constant \(C>0\) independent of \(t\). Since the background
metric is also independent of \(t\),
\[
\operatorname{sign}(\partial_t e_t)
=
\operatorname{sign}(\partial_t\varepsilon_t).
\]
Hence it suffices to prove pointwise monotonicity of \(\varepsilon_t\).

\section{Linearization, energy variation, and the rank-two Bochner identity}\label{sec:linearization}

\subsection{Variation formulas}\label{subsec:variation-formulas}

The first step is to differentiate Hitchin's equation along the real scaling orbit.

\begin{lemma}[Variation of the adjoint]\label{lem:adjoint}
With $s=\log t$ and $\Gamma_t=h_t^{-1}\partial_s h_t$, one has
\begin{equation}\label{eq:adjointvariation}
\partial_s A_t=A_t,
\qquad
\partial_s A_t^{\star_{h_t}}=A_t^{\star_{h_t}}+[A_t^{\star_{h_t}},\Gamma_t].
\end{equation}
\end{lemma}

\begin{proof}
The first identity follows directly from $\Phi_t=t\Phi$. In a fixed local frame,
\[
A_t^{\star_{h_t}}=h_t^{-1}A_t^{\dagger}h_t.
\]
Differentiating and using $\partial_s h_t=h_t\Gamma_t$ and $\partial_s h_t^{-1}=-\Gamma_t h_t^{-1}$ gives
\[
\partial_s A_t^{\star_{h_t}}
=-\Gamma_t A_t^{\star_{h_t}}+A_t^{\star_{h_t}}+A_t^{\star_{h_t}}\Gamma_t
=A_t^{\star_{h_t}}+[A_t^{\star_{h_t}},\Gamma_t].\qedhere
\]
\end{proof}

We now fix the sign convention carefully. Since
\[
F_{h_t}=-\partial_{\bar z}(h_t^{-1}\partial_z h_t)\,dz\wedge d\bar z
\]
and
\[
[\Phi_t,\Phi_t^{\star_{h_t}}]=[A_t,A_t^{\star_{h_t}}]dz\wedge d\bar z,
\]
Equation \eqref{eq:Hitchin} becomes
\begin{equation}\label{eq:Hitchinlocal}
D_{h_t}''(h_t^{-1}\partial_z h_t)=[A_t,A_t^{\star_{h_t}}].
\end{equation}
Differentiating gives the key linearized equation.

\begin{proposition}[Linearized Hitchin equation]\label{prop:linearized}
The endomorphism $\Gamma_t=h_t^{-1}\partial_s h_t$ satisfies
\begin{equation}\label{eq:linearized}
D_{h_t}''D_{h_t}'\Gamma_t
=2[A_t,A_t^{\star_{h_t}}]+[A_t,[A_t^{\star_{h_t}},\Gamma_t]].
\end{equation}
\end{proposition}

\begin{proof}
Differentiate \eqref{eq:Hitchinlocal}. In a holomorphic frame,
\[
\begin{aligned}
\partial_s(h_t^{-1}\partial_z h_t)
&=-\Gamma_t h_t^{-1}\partial_z h_t+h_t^{-1}\partial_z(h_t\Gamma_t)\\
&=\partial_z\Gamma_t+[h_t^{-1}\partial_z h_t,\Gamma_t]
=D_{h_t}'\Gamma_t.
\end{aligned}
\]
Since $D_{h_t}''=\bar\partial$ is independent of $t$, the derivative of the left-hand side of \eqref{eq:Hitchinlocal} is therefore $D_{h_t}''D_{h_t}'\Gamma_t$. On the other hand, by Lemma \ref{lem:adjoint},
\[
\begin{aligned}
\partial_s[A_t,A_t^{\star_{h_t}}]
&=[\partial_sA_t,A_t^{\star_{h_t}}]+[A_t,\partial_sA_t^{\star_{h_t}}]\\
&=[A_t,A_t^{\star_{h_t}}]+[A_t,A_t^{\star_{h_t}}+[A_t^{\star_{h_t}},\Gamma_t]]\\
&=2[A_t,A_t^{\star_{h_t}}]+[A_t,[A_t^{\star_{h_t}},\Gamma_t]].
\end{aligned}
\]
This proves \eqref{eq:linearized}.
\end{proof}

\subsection{Derivative of the pointwise energy density}\label{subsec:energy-derivative}

We next calculate the $s$-derivative of $\varepsilon_t=\tr\!\left(A_tA_t^{\star_{h_t}}\right)$.

\begin{proposition}[Energy derivative]\label{prop:energyderivative}
At any point $x\in X$, choose an $h_t$-unitary eigenbasis of the Hermitian self-adjoint endomorphism $\Gamma_t$, so that \eqref{eq:Gammadiag} holds and write
\begin{equation}\label{eq:Amatrix}
A_t=
\begin{pmatrix}
a&b\\
c&-a
\end{pmatrix}.
\end{equation}
Then
\begin{equation}\label{eq:energyderivative}
\partial_s\varepsilon_t
=4|a|^2+2(1+\gamma_t)|b|^2+2(1-\gamma_t)|c|^2.
\end{equation}
\end{proposition}

\begin{proof}
By Lemma \ref{lem:adjoint},
\[
\begin{aligned}
\partial_s\varepsilon_t
&=\tr((\partial_sA_t)A_t^{\star_{h_t}})+\tr(A_t\partial_sA_t^{\star_{h_t}})\\
&=2\tr\!\left(A_tA_t^{\star_{h_t}}\right)+\tr(A_t[A_t^{\star_{h_t}},\Gamma_t])\\
&=2\tr\!\left(A_tA_t^{\star_{h_t}}\right)+\tr([A_t,A_t^{\star_{h_t}}]\Gamma_t).
\end{aligned}
\]
In the basis \eqref{eq:Gammadiag} and \eqref{eq:Amatrix},
\[
\tr\!\left(A_tA_t^{\star_{h_t}}\right)=2|a|^2+|b|^2+|c|^2
\]
and
\[
[A_t,A_t^{\star_{h_t}}]_{11}=|b|^2-|c|^2.
\]
Since $[A_t,A_t^{\star_{h_t}}]$ is trace-free,
\[
\tr([A_t,A_t^{\star_{h_t}}]\Gamma_t)=2\gamma_t(|b|^2-|c|^2).
\]
Combining these identities gives \eqref{eq:energyderivative}.
\end{proof}

Thus the non-strict part of Theorem \ref{thm:mainSL2} follows from Proposition \ref{prop:Gammabound}. The remainder of the proof is devoted first to the estimate $\gamma_t\le1$ in Section~\ref{subsec:metric-variation-bound} and then to its equality case in Section~\ref{sec:equality}.

\subsection{The rank-two Bochner identity}\label{subsec:bochner-identity}

Let \[U:=\{x\in X:\gamma_t(x)>0\}.\]
 
 On $U$ the two eigenvalues $\gamma_t$ and $-\gamma_t$ of the
self-adjoint endomorphism $\Gamma_t$ are distinct. Hence $\gamma_t$
is smooth on $U$, and the corresponding eigenspaces form smooth
Hermitian line subbundles. We may therefore work locally in a smooth
\(h_t\)-unitary eigenframe and use the notation
\eqref{eq:Gammadiag} and \eqref{eq:Amatrix}.

Changing the unitary eigenframe only changes $b$ and $c$ by unit
complex phases; hence $|b|^2$ and $|c|^2$ are well-defined smooth
scalar functions on $U$.

 The standard first-variation formula for a simple eigenvalue then gives
\[
\partial_z\gamma_t=(D_{h_t}'\Gamma_t)_{11},\qquad -\partial_z\gamma_t=(D_{h_t}'\Gamma_t)_{22}.
\]
Throughout this section, $|B|^2:=\tr(BB^{\star_{h_t}})$ denotes the Hilbert--Schmidt norm induced by $h_t$.  Since
\[
\tr(\Gamma_t^2)=2\gamma_t^2,
\]
we can extract a differential equation for $\gamma_t$ from the linearized Hitchin equation \eqref{eq:linearized}.

\begin{lemma}[Basic Bochner formula]\label{lem:bochnerbasic}
At a point where $\gamma_t>0$,
\begin{equation}\label{eq:basicbochner}
2|\partial_z\gamma_t|^2+2\gamma_t\partial_{\bar z}\partial_z\gamma_t
=|D_{h_t}'\Gamma_t|^2+\tr\bigl(\Gamma_t D_{h_t}''D_{h_t}'\Gamma_t\bigr).
\end{equation}
Moreover,
\begin{equation}\label{eq:eigenineq}
|D_{h_t}'\Gamma_t|^2\ge 2|\partial_z\gamma_t|^2.
\end{equation}
\end{lemma}

\begin{proof}
Choose an $h_t$-unitary normal frame at the point, with its value chosen to diagonalize $\Gamma_t$. Since the Chern connection is unitary and $\Gamma_t=\Gamma_t^{\star_{h_t}}$,
\[
D_{h_t}''\Gamma_t=(D_{h_t}'\Gamma_t)^{\star_{h_t}}.
\]
Therefore
\[
\begin{aligned}
\partial_{\bar z}\partial_z\tr(\Gamma_t^2)
&=2\tr(D_{h_t}''\Gamma_t D_{h_t}'\Gamma_t)+2\tr(\Gamma_t D_{h_t}''D_{h_t}'\Gamma_t)\\
&=2|D_{h_t}'\Gamma_t|^2+2\tr(\Gamma_t D_{h_t}''D_{h_t}'\Gamma_t).
\end{aligned}
\]
The left side is $\partial_{\bar z}\partial_z(2\gamma_t^2)=4|\partial_z\gamma_t|^2+4\gamma_t\partial_{\bar z}\partial_z\gamma_t$, proving \eqref{eq:basicbochner}.

For \eqref{eq:eigenineq}, in the same frame the diagonal entries of $D_{h_t}'\Gamma_t$ are $\partial_z\gamma_t$ and $-\partial_z\gamma_t$. The off-diagonal entries only add nonnegative terms to the Hilbert--Schmidt norm, so
\[
|D_{h_t}'\Gamma_t|^2\ge |\partial_z\gamma_t|^2+|-\partial_z\gamma_t|^2=2|\partial_z\gamma_t|^2.\qedhere
\]
\end{proof}

The decisive rank-two algebra is the following.

\begin{lemma}\label{lem:rank2algebra}
At a point where $\Gamma_t$ and $A_t$ have the forms \eqref{eq:Gammadiag}--\eqref{eq:Amatrix},
\begin{equation}\label{eq:11entry}
(D_{h_t}''D_{h_t}'\Gamma_t)_{11}
=2\bigl((1+\gamma_t)|b|^2+(\gamma_t-1)|c|^2\bigr),
\end{equation}
and consequently
\begin{equation}\label{eq:traceGamma}
\tr\bigl(\Gamma_t D_{h_t}''D_{h_t}'\Gamma_t\bigr)
=4\gamma_t\bigl((1+\gamma_t)|b|^2+(\gamma_t-1)|c|^2\bigr).
\end{equation}
\end{lemma}

\begin{proof}
In the chosen $h_t$-unitary eigenframe,
\[
A_t^{\star_{h_t}}=
\begin{pmatrix}
\bar a&\bar c\\
\bar b&-\bar a
\end{pmatrix}.
\]
A direct calculation gives
\[
[A_t,A_t^{\star_{h_t}}]_{11}=|b|^2-|c|^2.
\]
Moreover,
\[
[A_t^{\star_{h_t}},\Gamma_t]=
\begin{pmatrix}
0&-2\gamma_t\bar c\\
2\gamma_t\bar b&0
\end{pmatrix},
\]
and therefore
\[
[A_t,[A_t^{\star_{h_t}},\Gamma_t]]_{11}=2\gamma_t(|b|^2+|c|^2).
\]
Thus
\[
\bigl(2[A_t,A_t^{\star_{h_t}}]+[A_t,[A_t^{\star_{h_t}},\Gamma_t]]\bigr)_{11}
=2\bigl((1+\gamma_t)|b|^2+(\gamma_t-1)|c|^2\bigr).
\]
Taking the $(1,1)$ entry of \eqref{eq:linearized} yields \eqref{eq:11entry}.

The matrix $D_{h_t}''D_{h_t}'\Gamma_t$ is trace-free because $\Gamma_t$ is trace-free. Therefore its $(2,2)$ entry is the negative of its $(1,1)$ entry, and
\[
\tr(\Gamma_t D_{h_t}''D_{h_t}'\Gamma_t)=2\gamma_t(D_{h_t}''D_{h_t}'\Gamma_t)_{11},
\]
which gives \eqref{eq:traceGamma}.
\end{proof}

Combining the previous two lemmas gives the scalar identity that drives the proof.

\begin{proposition}[Rank-two eigenvalue equation]\label{prop:eigenvalueequation}
On the open set $\{\gamma_t>0\}$,
\begin{equation}\label{eq:eigenvalueequation}
\partial_{\bar z}\partial_z\gamma_t
=\frac{|D_{h_t}'\Gamma_t|^2-2|\partial_z\gamma_t|^2}{2\gamma_t}
+2(1+\gamma_t)|b|^2
+2(\gamma_t-1)|c|^2.
\end{equation}
In particular,
\begin{equation}\label{eq:eigenvalueineq}
\partial_{\bar z}\partial_z\gamma_t
\ge 2(1+\gamma_t)|b|^2+2(\gamma_t-1)|c|^2.
\end{equation}
\end{proposition}

\begin{proof}
Substitute \eqref{eq:traceGamma} into \eqref{eq:basicbochner} and divide by $4\gamma_t$. The inequality follows from \eqref{eq:eigenineq}.
\end{proof}

\section{Proof of Theorem~\ref{thm:mainSL2}}\label{sec:maximum-principle}

\subsection{The bound \texorpdfstring{$\|\Gamma_t\|_{\op}\le1$} {on the metric variation}}\label{subsec:metric-variation-bound}

\begin{proof}[Proof of Proposition \ref{prop:Gammabound}]
Let $\Omega:=\{x\in X:\gamma_t(x)>1\}.$ On $\Omega$, Equation \eqref{eq:eigenvalueineq} gives
\[
\partial_{\bar z}\partial_z\gamma_t\ge0,
\]
because both coefficients $1+\gamma_t$ and $\gamma_t-1$ are positive. Hence $\gamma_t$ is subharmonic on every connected component of $\Omega$.

Suppose first that a connected component $\Omega_0$ is a proper subset of $X$. Since $X$ is compact, $\overline{\Omega_0}$ is compact. By continuity,
\[
\gamma_t=1\quad\text{on }\partial\Omega_0,
\qquad
\gamma_t>1\quad\text{in }\Omega_0.
\]
Thus the maximum of $\gamma_t$ on $\overline{\Omega_0}$ is strictly larger than $1$ and is attained at an interior point. The strong maximum principle contradicts subharmonicity unless $\gamma_t$ is constant on $\Omega_0$; but a constant value $>1$ cannot extend continuously to the boundary value $1$. Hence no proper component exists.

If $\Omega$ is non-empty, the only remaining possibility is $\Omega=X$. Then $\gamma_t$ is globally subharmonic on the compact surface $X$, so $\gamma_t$ is constant. Equation \eqref{eq:eigenvalueequation} is now a sum of nonnegative terms, because $\gamma_t>1$. Therefore
\[
b=c=0,
\qquad
D_{h_t}'\Gamma_t=0.
\]
Since $\Gamma_t$ is self-adjoint, $D_{h_t}''\Gamma_t=(D_{h_t}'\Gamma_t)^{\star_{h_t}}=0$ as well. Thus $\Gamma_t$ is $D_{h_t}$-parallel. Its two distinct eigenvalues define a $D_{h_t}$-parallel orthogonal splitting
\[
E=L_+\oplus L_-.
\]
In particular, $L_\pm$ are holomorphic line subbundles. The identities $b=c=0$ say that $\Phi_t$ preserves both line subbundles; equivalently, since $t>0$, $\Phi$ preserves them. Moreover $[A_t,A_t^{\star_{h_t}}]=0$, so Hitchin's equation gives $F_{h_t}=0$. Hence each $L_\pm$ has degree zero. This contradicts stability.

Therefore $\Omega=\varnothing$, i.e., we have $\gamma_t\le1$ everywhere. Since the eigenvalues of $\Gamma_t$ are $\gamma_t$ and $-\gamma_t$, this is exactly $\|\Gamma_t\|_{\op}\le1.$
\end{proof}

Combining Proposition \ref{prop:Gammabound} with Proposition \ref{prop:energyderivative} proves the non-strict part of Theorem \ref{thm:mainSL2}.

\subsection{The equality case and strict monotonicity}\label{sec:equality}

The estimate $\gamma_t\le1$ is sharp. The boundary case $\gamma_t=1$ corresponds precisely to the Hodge-theoretic fixed-point phenomenon.

\begin{proposition}\label{prop:equality}
Assume $(E,\Phi)$ is stable and that $\gamma_t(x_0)=1$ at some point $x_0\in X$. Then $\gamma_t\equiv1$ and $(E,\Phi)$ is a $\C^*$-fixed Higgs bundle.
\end{proposition}

\begin{proof}
Since $\gamma_t(x_0)=1$, we work in the connected component of $\{\gamma_t>0\}$ containing $x_0$. Set
\[
u:=\gamma_t-1\le0.
\]
From \eqref{eq:eigenvalueineq},
\[
\begin{aligned}
\partial_{\bar z}\partial_z u
&\ge 2(1+\gamma_t)|b|^2+2u|c|^2,
\end{aligned}
\]
so
\begin{equation}\label{eq:uineq}
\partial_{\bar z}\partial_z u-2|c|^2u\ge2(1+\gamma_t)|b|^2\ge0.
\end{equation}
The zeroth-order coefficient in the operator on the left is nonpositive. Since $u\le0$ attains its maximum $0$ at $x_0$, the strong maximum principle gives $u\equiv0$ on that component. The component cannot have boundary in $\{\gamma_t=0\}$, so $\gamma_t\equiv1$ on $X$.

Returning to the exact identity \eqref{eq:eigenvalueequation}, we find
\[
b=0,
\qquad
D_{h_t}\Gamma_t=0.
\]
Thus the eigenspaces of $\Gamma_t$ give a global $D_{h_t}$-parallel, hence holomorphic, splitting $E=L_+\oplus L_-$. In that splitting
\[
A_t=
\begin{pmatrix}
a&0\\
c&-a
\end{pmatrix}.
\]
Because $D_{h_t}\Gamma_t=0$, the left side of the linearized equation \eqref{eq:linearized} vanishes. With $\gamma_t=1$ and $b=0$, a direct multiplication shows that the $(2,1)$ entry of
\[
2[A_t,A_t^{\star_{h_t}}]+[A_t,[A_t^{\star_{h_t}},\Gamma_t]]
\]
is $4\bar a\,c$. Hence $4\bar a\,c=0.$

If $c\equiv0$, then $\Phi$ is diagonal and both $L_+$ and $L_-$ are $\Phi$-invariant. Since $\deg L_++\deg L_-=0$, stability is impossible. Hence $c\not\equiv0$. On the nonempty open set where $c\ne0$ we have $a=0$, and holomorphicity then forces $a\equiv0$. Thus, with respect to the holomorphic splitting \(E=L_+\oplus L_-\), we have
\[
\Phi_t=
\begin{pmatrix}
0&0\\
\beta_t&0
\end{pmatrix}.
\]
By taking $t=1$, we obtain that $(E,\Phi)$ is a $\C^*$-fixed Higgs bundle.
\end{proof}

\begin{corollary}\label{cor:strict}
If $(E,\Phi)$ is not $\C^*$-fixed, then on $X$ we have $\gamma_t<1$. Consequently,
\[
\partial_t e_t(x)>0
\qquad\text{whenever }\Phi(x)\ne0.
\]
\end{corollary}

\begin{proof}
Proposition \ref{prop:equality} shows that $\gamma_t=1$ anywhere would force the bundle to be fixed. Hence $\gamma_t<1$. In \eqref{eq:energyderivative}, the three coefficients
\[
4,\qquad2(1+\gamma_t),\qquad2(1-\gamma_t)
\]
are then strictly positive. Thus $\partial_s\varepsilon_t(x)>0$ whenever $A_t(x)\ne0$.
\end{proof}

\begin{remark}
For a nonzero two-step Hodge fixed point, one can choose the standard grading normalization so that the logarithmic metric variation has eigenvalues $\{1,-1\}$; hence the bound in Proposition \ref{prop:Gammabound} is sharp.  
\end{remark}

\section{Failure in every rank
\texorpdfstring{$n\geq 3$}
{n >= 3}}\label{sec:highrank-counterexample}

We now prove Theorem~\ref{thm:highrank-counterexample}.  

\subsection{A family of Higgs bundles}\label{sec:familyHiggs}

Let $X$ be a compact Riemann surface of genus $g\ge2$.  Choose an odd theta
characteristic $L$, so that
\[
 L^2\simeq K_X,
 \qquad
 h^0(X,L)\equiv1\pmod2.
\]
In particular $H^0(X,L)\ne0$.  Fix a nonzero section
\[\sigma\in H^0(X,L).\]  Since $\deg L=g-1>0$, the zero divisor of $\sigma$ is nonempty;
fix a point $p$ such that $\sigma(p)=0$.  Fix an isomorphism $L^2\simeq K_X$.  Via the
induced identifications
\[
 K_XL^{-2}\simeq\mathcal O_X,
 \qquad
 K_XL^{-1}\simeq L,
\]
we may regard $\sigma$ as a section of $K_XL^{-1}$ and $\sigma^2$ as a section of
$K_X$.

For an integer $n\ge3$, define
\[
 E_n:=L^{-1}\oplus\mathcal O^{\oplus(n-2)}\oplus L.
\]
Thus $\det E_n\simeq\mathcal O$.  With respect to this ordered decomposition, for any $u\in\C$, define a Higgs field $\Psi_{n,u}$ by
\[
 (\Psi_{n,u})_{ij}
 :=
 \begin{cases}
 \sigma, & (i,j)=(1,2)\ \text{or}\ (n-1,n),\\
 \sigma^2, & j=i+1,\quad 2\le i\le n-2,\\
 u, & (i,j)=(1,n),\\
 0, & \text{otherwise},
 \end{cases}
\]
where the $(1,n)$ entry is interpreted through
$K_XL^{-2}\simeq\mathcal O_X$.  When $n=3$, the middle line involving $\sigma^2$
is absent.  Set
\[
 (E_n,\Phi_n):=(E_n,\Psi_{n,1}).
\]
The field $\Psi_{n,u}$ is trace-free, so $(E_n,\Psi_{n,u})$ is an
$\mathrm{SL}(n,\C)$ Higgs bundle.  Since $\sigma(p)=0$, all entries of
$\Psi_{n,u}(p)$ vanish except the $(1,n)$ entry $u$.  Hence
\[
 \Psi_{n,0}(p)=0,
 \qquad
 \Phi_n(p)\ne0.
\]

\begin{lemma}\label{lem:highrank-stability}
For every $u\in\C$, the Higgs bundle $(E_n,\Psi_{n,u})$ is stable.
\end{lemma}

\begin{proof}
At the generic point of $X$, after trivializing $K_X$, the Higgs field
$\Psi_{n,u}$ is represented by a regular nilpotent endomorphism.  Indeed,
$\Psi_{n,u}^n=0$, while the product of the adjacent arrows gives
\[
 (\Psi_{n,u}^{\,n-1})_{1n}=\sigma^{2n-4}\ne0.
\]
Thus the generic fibre consists of a single nilpotent Jordan block, whose
unique invariant subspace of dimension $r$ is $\ker\Psi_{n,u}^{\,r}$ for
$1\le r\le n-1$.

For
\[
 F_r:=L^{-1}\oplus\mathcal O^{\oplus(r-1)}\subset E_n,
 \qquad 1\le r\le n-1,
\]
the strictly upper-triangular form of $\Psi_{n,u}$ shows that $F_r$ is
$\Psi_{n,u}$-invariant and is generically equal to
$\ker\Psi_{n,u}^{\,r}$.  Let $F\subset E_n$ be a saturated
$\Psi_{n,u}$-invariant subsheaf of rank $r$.  Its generic fibre must therefore
coincide with that of $F_r$.  The induced map $F\to E_n/F_r$ vanishes
generically and hence vanishes identically because $E_n/F_r$ is torsion-free.
Thus $F\subset F_r$.  Since $F_r/F$ is torsion and injects into the
torsion-free sheaf $E_n/F$, saturation of $F$ implies $F=F_r$.

Finally,
\[
 \deg F_r=-\deg L=-(g-1),
 \qquad
 \mu(F_r)=-\frac{g-1}{r}<0=\mu(E_n).
\]
The saturation of an invariant subsheaf is again invariant and can only
increase degree, so the preceding calculation proves stability.
\end{proof}

By definition, the bundle $(E_n,\Psi_{n,0})$ is $\C^*$-fixed.  

\subsection{The scaling identity}

For $t>0$ set
\[
 g_t:=\operatorname{diag}\left(
 t^{1-\frac{n+1}{2}},
 t^{2-\frac{n+1}{2}},
 \ldots,
 t^{n-\frac{n+1}{2}}
 \right).
\]
The sum of the exponents is zero, so $\det g_t=1$.  Since $t>0$, the real
powers are unambiguous even when $n$ is even.  For every adjacent entry one has
\[
 t\,\frac{(g_t)_{ii}}{(g_t)_{i+1,i+1}}=1,
\]
whereas for the long $(1,n)$ entry,
\[
 t\,\frac{(g_t)_{11}}{(g_t)_{nn}}=t^{2-n}.
\]
Therefore
\begin{equation}\label{eq:highrank-scaling-identity}
 g_t(t\Phi_n)g_t^{-1}
 =\Psi_{n,t^{-(n-2)}}.
\end{equation}
Consequently the positive real scaling orbit $t\cdot(E_n,\Phi_n)$ converges in the Higgs-bundle moduli space to the stable fixed point
$(E_n,\Psi_{n,0})$ as $t\to+\infty$.

Let $k_{n,u}$ denote the normalized harmonic metric of
$(E_n,\Psi_{n,u})$.  By uniqueness and
\eqref{eq:highrank-scaling-identity}, if $h_{n,t}$ is the normalized harmonic
metric of $(E_n,t\Phi_n)$, then
\[
 h_{n,t}=g_t^*k_{n,t^{-(n-2)}}.
\]
Thus $g_t:(E_n,h_{n,t})\to(E_n,k_{n,t^{-(n-2)}})$ is an isometry
intertwining the Higgs fields.  Hence
\begin{equation}\label{eq:highrank-energy-invariance}
 e_{n,t}(x)
 =e_{(E_n,\Psi_{n,t^{-(n-2)}})}(x),
\end{equation}
where the energy density on the right is computed using the harmonic metric
$k_{n,t^{-(n-2)}}$.

\subsection{Smooth dependence of the harmonic metric}

\begin{lemma}\label{lem:highrank-smoothness}
The normalized harmonic metrics $k_{n,u}$ associated with
$(E_n,\Psi_{n,u})$ depend smoothly on $(\Re u,\Im u)$ for $u$ near $0$,
with smooth dependence in the spatial variable as well.
\end{lemma}

\begin{proof}
By Lemma~\ref{lem:highrank-stability}, $(E_n,\Psi_{n,u})$ is stable for every
$u$ near $0$.  After fixing the determinant metric, the corresponding
harmonic metric is unique.  The asserted parameter dependence is the standard
smooth-dependence consequence.
\end{proof}

\subsection{Energy collapse and asymptotics}

For real $u$ near $0$, the only nonzero entry of $\Psi_{n,u}(p)$ is the
$(1,n)$ entry $u$, interpreted through $K_XL^{-2}\simeq\mathcal O_X$.  Since
the energy density is quadratic in the Higgs field and $k_{n,u}$ depends
smoothly on $u$ by Lemma~\ref{lem:highrank-smoothness}, there is a smooth
positive function $a_n(u)$ near $0$ such that
\[
 e_{(E_n,\Psi_{n,u})}(p)=u^2a_n(u),
 \qquad
 a_{0,n}:=a_n(0)>0.
\]
Substituting $u=t^{-(n-2)}$ into
\eqref{eq:highrank-energy-invariance} gives the exact identity
\begin{equation}\label{eq:highrank-energy-exact}
 e_{n,t}(p)
 =t^{-2(n-2)}a_n\!\left(t^{-(n-2)}\right)
\end{equation}
for all sufficiently large $t$.  In particular,
\[
 e_{n,t}(p)\longrightarrow0,
\]
while $e_{n,1}(p)>0$ because $\Phi_n(p)\ne0$.  Hence pointwise
nondecreasing energy already fails.

Taylor expansion of $a_n$ at zero gives
\[
 a_n\!\left(t^{-(n-2)}\right)
 =a_{0,n}+O\!\left(t^{-(n-2)}\right),
\]
and therefore \eqref{eq:highrank-energy-exact} gives
\[
 e_{n,t}(p)
 =a_{0,n}t^{-2(n-2)}+O\!\left(t^{-3(n-2)}\right),
\]
which is \eqref{eq:highrank-limits}.  Differentiating
\eqref{eq:highrank-energy-exact} with respect to $t$ gives
\begin{equation}\label{eq:highrank-energy-derivative}
 \begin{aligned}
 \partial_t e_{n,t}(p)
 & =-(n-2)t^{-2(n-2)-1}\\
 &\qquad\cdot\left[
 2a_n\!\left(t^{-(n-2)}\right)
 +t^{-(n-2)}a_n'\!\left(t^{-(n-2)}\right)
 \right].
 \end{aligned}
\end{equation}
Since $a_n$ is smooth near zero,
\[
 a_n\!\left(t^{-(n-2)}\right)=a_{0,n}+O\!\left(t^{-(n-2)}\right),
 \qquad
 a_n'\!\left(t^{-(n-2)}\right)=O(1).
\]
Substituting these estimates into \eqref{eq:highrank-energy-derivative} gives
\[
 \partial_t e_{n,t}(p)
 =-2(n-2)a_{0,n}t^{-2(n-2)-1}
 +O\!\left(t^{-3(n-2)-1}\right),
\]
which is \eqref{eq:highrank-derivative-limit}.  Since $a_{0,n}>0$, it follows
that
\[
 \partial_t e_{n,t}(p)<0
\]
for all sufficiently large $t$.

Finally, $(E_n,\Phi_n)$ is not $\C^*$-fixed.  Otherwise, the
corresponding pointwise energy density would be independent of $t$. This completes
the proof of Theorem~\ref{thm:highrank-counterexample}.

\subsection{Failure of Conjecture 10.3 in \cite{LiSurvey}}\label{subsec:Li-conjecture-10-3}
Recall the conjectured estimate \eqref{eq:Liconj} of pointwise energy as $t\to0$.

Let $\Theta_n$ be the Higgs field on $E_n$ whose only nonzero entry, with
respect to the decomposition of $E_n$ above, is
\[
 (\Theta_n)_{1n}:=1,
\]
where, as before, the entry is interpreted through
$K_XL^{-2}\simeq\mathcal O_X$.  We first identify the zero-limit of the
$\C^*$-orbit of $(E_n,\Phi_n)$.  For $t>0$, set
\[
q_t:=\operatorname{diag}\left(t^{-1/2},1,\ldots,1,t^{1/2}\right).
\]
Then $q_t$ is a holomorphic automorphism of $E_n$ with determinant one, and
\[
 \bigl(q_t(t\Phi_n)q_t^{-1}\bigr)_{ij}
 =
 \begin{cases}
 t^{1/2}\sigma, & (i,j)=(1,2)\ \text{or}\ (n-1,n),\\
 t\sigma^2, & j=i+1,\quad 2\le i\le n-2,\\
 1, & (i,j)=(1,n),\\
 0, & \text{otherwise}.
 \end{cases}
\]
Thus
\begin{equation*}
 \lim_{t\to0} t\cdot[(E_n,\Phi_n)]
 =[(E_n,\Theta_n)].
\end{equation*}
The limiting Higgs bundle is polystable.  Indeed,
\[
 (E_n,\Theta_n)
 \simeq
 \left(
 L^{-1}\oplus L,
 \begin{pmatrix}0&1\\0&0\end{pmatrix}
 \right)
 \oplus(\mathcal O,0)^{\oplus(n-2)}.
\]

Let $\widehat e_n$ denote the energy density associated with
$(E_n,\Theta_n)$.  Since $(\Theta_n)_{1n}=1$, we have
$\Theta_n(p)\ne0$, and therefore
\begin{equation}\label{eq:zero-limit-positive-energy}
 \widehat e_n(p)>0.
\end{equation}
On the other hand, \eqref{eq:highrank-energy-exact} gives
\begin{equation}\label{eq:highrank-infinity-zero-energy}
 e_{n,t}(p)\longrightarrow0
 \qquad\text{as }t\to+\infty.
\end{equation}
Choose $T>0$ sufficiently large that
\[
 e_{n,T}(p)<\widehat e_n(p).
\]
For the stable Higgs bundle $(E_n,T\Phi_n)$, its zero-limit is still
$[(E_n,\Theta_n)]$. Consequently
\[
 e_{n,T}(p)
 <
 e_{(E_n,\Theta_n)}(p),
\]
which is the strict reverse of the inequality predicted by
\eqref{eq:Liconj}.


\begin{thebibliography}{99}

\bibitem{Corlette1988}
K. Corlette,
\emph{Flat $G$-bundles with canonical metrics},
J. Differential Geom. \textbf{28} (1988), no. 3, 361--382.
DOI: \href{https://doi.org/10.4310/jdg/1214442469}{\nolinkurl{10.4310/jdg/1214442469}}.

\bibitem{DaiLiCyclic}
S. Dai and Q. Li,
\emph{On cyclic Higgs bundles},
Math. Ann. \textbf{376} (2020), no. 3--4, 1225--1260.
DOI: \href{https://doi.org/10.1007/s00208-018-1779-4}{\nolinkurl{10.1007/s00208-018-1779-4}}.

\bibitem{DaiLiDomination}
S. Dai and Q. Li,
\emph{Domination results in $n$-Fuchsian fibers in the moduli space of Higgs bundles},
Proc. London Math. Soc. (3) \textbf{124} (2022), no. 4, 427--477.
DOI: \href{https://doi.org/10.1112/plms.12431}{\nolinkurl{10.1112/plms.12431}}.

\bibitem{DeroinTholozan}
B. Deroin and N. Tholozan,
\emph{Dominating surface group representations by Fuchsian ones},
Int. Math. Res. Not. IMRN (2016), no. 13, 4145--4166.
DOI: \href{https://doi.org/10.1093/imrn/rnv275}{\nolinkurl{10.1093/imrn/rnv275}}.

\bibitem{Donaldson1987}
S. K. Donaldson,
\emph{Twisted harmonic maps and the self-duality equations},
Proc. London Math. Soc. (3) \textbf{55} (1987), no. 1, 127--131.
DOI: \href{https://doi.org/10.1112/plms/s3-55.1.127}{\nolinkurl{10.1112/plms/s3-55.1.127}}.

\bibitem{Hitchin1987}
N. J. Hitchin,
\emph{The self-duality equations on a Riemann surface},
Proc. London Math. Soc. (3) \textbf{55} (1987), no. 1, 59--126.
DOI: \href{https://doi.org/10.1112/plms/s3-55.1.59}{\nolinkurl{10.1112/plms/s3-55.1.59}}.

\bibitem{Hitchin1992}
N. J. Hitchin,
\emph{Lie groups and Teichm\"uller space},
Topology \textbf{31} (1992), no. 3, 449--473.
DOI: \href{https://doi.org/10.1016/0040-9383(92)90044-I}{\nolinkurl{10.1016/0040-9383(92)90044-I}}.

\bibitem{LiHarmonic}
Q. Li,
\emph{Harmonic maps for Hitchin representations},
Geom. Funct. Anal. \textbf{29} (2019), no. 2, 539--560.
DOI: \href{https://doi.org/10.1007/s00039-019-00491-7}{\nolinkurl{10.1007/s00039-019-00491-7}}.


\bibitem{LiSurvey}
Q. Li,
\emph{An introduction to Higgs bundles via harmonic maps},
SIGMA \textbf{15} (2019), 035, 30 pp.
DOI: \href{https://doi.org/10.3842/SIGMA.2019.035}{\nolinkurl{10.3842/SIGMA.2019.035}}.

\bibitem{SagmanSmillie}
N. Sagman and P. Smillie,
\emph{Local asymptotics for Hitchin's equations and high energy harmonic maps},
Math. Ann. \textbf{394} (2026), article no. 17.
DOI: \href{https://doi.org/10.1007/s00208-026-03375-y}{\nolinkurl{10.1007/s00208-026-03375-y}}.

\bibitem{SagmanTosic}
N. Sagman and O. To\v{s}i\'c,
\emph{On Hitchin's equations for cyclic $G$-Higgs bundles},
Adv. Math. \textbf{482} (2025), part A, article no. 110599.
DOI: \href{https://doi.org/10.1016/j.aim.2025.110599}{\nolinkurl{10.1016/j.aim.2025.110599}}.


\bibitem{Simpson1992}
C. T. Simpson,
\emph{Higgs bundles and local systems},
Publ. Math. Inst. Hautes \'Etudes Sci. \textbf{75} (1992), 5--95.
DOI: \href{https://doi.org/10.1007/BF02699491}{\nolinkurl{10.1007/BF02699491}}.

\bibitem{UhlenbeckYau1986}
K. Uhlenbeck and S.-T. Yau,
\emph{On the existence of Hermitian--Yang--Mills connections in stable vector bundles},
Comm. Pure Appl. Math. \textbf{39} (1986), suppl. S1, S257--S293.
DOI: \href{https://doi.org/10.1002/cpa.3160390714}{\nolinkurl{10.1002/cpa.3160390714}}.

\end{thebibliography}
\end{document}